\documentclass{amsart}
\usepackage{inputenc, amsmath,amssymb,graphicx,verbatim, verbatim, bbm}
\usepackage[hidelinks]{hyperref}
\usepackage{blindtext}
\usepackage{amsthm}
\usepackage{tikz-cd}
\usepackage{tikz}
\usetikzlibrary{calc}
\usepackage{mathtools}
\usepackage{soul, scalerel}
\usepackage{multirow}
\usepackage[margin=1.25in]{geometry}
\usepackage{url}

\newtheorem{theorem}{Theorem}
\newtheorem{defn}{Definition}
\newtheorem{lem}[theorem]{Lemma}

\newtheorem{cor}[theorem]{Corollary}

\theoremstyle{definition}
\newtheorem{rmk}{Remark}

\newcommand{\F}{\mathbb F}
\newcommand{\N}{\mathbb N}
\newcommand{\Z}{\mathbb Z}
\newcommand{\R}{\mathbb R}

\title[An inverse theorem for GAP]{An inverse theorem for Generalized Arithmetic Progression with mild multiplicative property}
\author{Ernie Croot and Junzhe Mao}
\address{School of Mathematics,
Georgia Institute of Technology,
Atlanta, GA 30332-0160 USA} 
\email{ecroot@math.gatech.edu}
\address{School of Mathematics,
Georgia Institute of Technology,
Atlanta, GA 30332-0160 USA}
\email{jmao87@gatech.edu}

\date{}
\begin{document}
\begin{abstract}
    In this paper we prove a structural theorem for generalized arithmetic progressions in $\F_p$ which contain a large product set of two other progressions.
\end{abstract}
\maketitle

\section{Introduction}
Let $p$ be a prime number, $\F_p$ be the finite field with $p$ elements. Recall that a generalized arithmetic progression (GAP) of rank $d$ in $\F_p$ is a set of the form \[B=x_0+\{\sum_{i=1}^da_ix_i:a_i\in[0,N_i-1] \}\]for some elements $x_0,x_1,\ldots,x_d\in\F_p$ and positive integers $2\leq N_1,\ldots,N_d\leq p.$ We say $B$ is proper if $|B|=N_1N_2\ldots N_d.$ GAPs play an important role in additive combinatorics, as they exhibit an almost periodic behavior under addition. Another example of additively structured sets is Bohr sets. Given $\Gamma\subset\F_p$ and $\epsilon>0,$ the Bohr set $B(\Gamma,\epsilon)$ is defined as \[B(\Gamma,\epsilon)=\left\{x\in\F_p:\left\|\frac{xr}{p}\right\|<\epsilon,\forall r\in\Gamma\right\}.\]Here $\left\|\cdot\right\|$ denotes the distance to the nearest integer. One can think of Bohr sets as approximate subspaces or (additive) subgroups in $\F_p$. It turns out that GAPs and Bohr sets are closely related. Using the geometry of numbers, one can show that any Bohr set $B(\Gamma,\epsilon)$ must contain a large symmetric proper GAP with bounded rank. This fact has been used to prove a foundational result in additive combinatorics, namely Freiman's theorem: given a finite set $A\subset\F_p$  and some $C>0,$ if $|A+A|\leq C|A|,$ where\[A+A=\{a_1+a_2:a_1,a_2\in A\}, \]then there exists a  GAP $B$ such that $A\subset B$ and $|B|\ll_C|A|$. Here we use the Vinogradov notation $\ll_C$ meaning that there is some constant $K>0$ depending on $C$ so that $|B|\leq K|A|.$ Freiman's theorem tells us that GAPs are the ``standard model'' for sets with rich additive structures. For more properties and applications of Bohr sets and GAPs, we refer the reader to \cite{Tao:2006aa}.  

One of the reasons people are concerned about these additively structured sets is the sum-product phenomenon, which basically says no set can have rich additive structure and multiplicative structure simultaneously. For example, in \cite{Erdos:1983aa} Erd{\"o}s and Szemer{\'e}di conjectured that if a set $A\subset\Z$ satisfies \[|A+A|\leq K|A| \]for some constant $K>0,$ then for any $\epsilon>0,$ $|AA|\gg|A|^{2-\epsilon}$, where \[AA=\{a_1a_2:a_1,a_2\in A\}. \] This has been confirmed for $A\subset\R$ by a result of Elekes and Ruzsa \cite{Elekes:2003aa}, and recently verified for thin sets $A\subset\F_p$ by a result of Kerr, Mello and Shparlinski \cite{Kerr:2024aa}. The case when $A\subset\F_p$ and $|A|\sim p^{1/2}$ remains open. We refer the reader to \cite{Mohammadi:2023aa,Rudnev:2020aa} for some relevant results.
To prove such an estimate, a common method is to bound the multiplicative energy $E(A)$ of $A,$ defined as the number of solutions to the equation \[a_1a_2=a_3a_4,\ a_1,a_2,a_3,a_4\in A. \]By Freiman's theorem, this reduces to bounding the multiplicative energy of GAPs in $\F_p.$ Better bounds on the multiplicative energy also have many applications in number theory. For example, Chang \cite{Chang:2008aa} used a nontrivial estimate for the multiplicative energy to give a Burgess-type bound for character sums over proper GAPs. Following a similar approach Hanson \cite{Hanson:2015aa} generalized this to character sums over regular Bohr sets. In fact, if one can prove the optimal bound \[E(B)\ll|B|^{2+o(1)} \]for any proper GAP $B\subset\F_p,$ then for any nonprincipal Dirichlet character $\chi$ of $\F_p,$ Burgess's bound can be fully generalized to \[\sum_{n\in B}\chi(n)=O(p^{-\delta}|B|) \]whenever $|B|\geq p^{1/4+\epsilon}.$ With current techniques one can prove near-optimal bounds on the multiplicative energy of GAPs with some extra structure. For example, Kerr \cite{Kerr:2021aa} proved the following result:
\begin{theorem}[Kerr]
	Let $P\subset \F_p$ be a GAP given by \[P=x_0+\{\sum_{i=1}^da_ix_i:a_i\in[1,N]\} \]and suppose \[P^\prime=\{\sum_{i=1}^da_ix_i:|a_i|\leq N^2\} \]is proper. Then $E(P)\ll|P|^2(\log N)^{2d+1}.$
\end{theorem}
The above discussion suggests that we can only hope for some mild multiplicative structure in a GAP or Bohr set. Let's consider \[B=\{a+bt:|a|,|b|\leq N\}, \]which is a symmetric GAP. Assume $t$ satisfies a quadratic polynomial equation mod $p$\[c_0+c_1t+t^2=0 \]with $c_i\in\Z$ bounded by some constant $C$ independent of $p$ . Let $A=\{a^\prime+b^\prime t:|a^\prime|,|b^\prime|\leq\frac{1}{2C} N^{1/2}\},$ then obviously $A$ satisfies the condition of Kerr's theorem. As a result $E(A)$ is small and $A$ does not possess a rich multiplicative structure. On the other hand, it's easy to check that $AA\subset B.$ This set $B$ gives us an example of a GAP with some mild multiplicative properties, i.e. containing a large product set. It raises the following question: if a GAP $B$ contains a product set $AA$ with $|B|^{1/2} \ll |A| \ll |B|^{1/2}$ and 
$A$ also a GAP, what can we say about $B$? In this paper we show that under some natural assumptions on $A$ and $B,$ the generators of $B$ must be given by the roots of some low-height, low-degree polynomials.

$\textbf{Notations.}$  For any element $x\in\F_p,$ we define its integer height\[|x|=\min{\{|a|:a\in\Z,\ a\equiv x\pmod p\}}. \]
For a polynomial $f\in\Z[x],$ let ${\overline f}(x)$ 
denote the projection of $f(x)$ into ${\mathbb F}_p[x]$, let $H(f)$ denote the height of $f,$ which is the maximum of the absolute value of coefficients of $f$. We use $M_n(\F_p)$ to denote the ring of $n\times n$ matrices over $\F_p.$

For any $x>0,$ let $[x]=\{1,2,\ldots,\lfloor x\rfloor\},$ where $\lfloor x\rfloor$ denote the integer part of $x.$

Given $x\in\F_p,A\subset\F_p,$ we write the dilation of $A$ by $x$ as $x\cdot A=\{xa:a\in A\}.$

For any $n\in\N, A\subset\F_p,$ we write the iterated sumset as \[\begin{matrix}
	nA=\underbrace{A+A+\cdots+A}\\ \hspace{9mm}n
\end{matrix}\]

\section{Main Results}
We first consider a simple case, namely a proper GAP $B$ of rank $2$ and size at most $o(p^{2/3})$. We show that if $B$ contains the product set of a sub-GAP $A$ of size $|A|\approx |B|^{1/2}$, then the generators of $B$ must be ``algebraic'' in the sense that they are the roots of some low-height quadratic polynomials.
\begin{theorem}
	Let $0<c<1,N=N(p)\in\N$ with $\lim_{p\to \infty}N(p)=\infty$, $N = o(p^{1/3})$, $B=\{a+bt:|a|,|b|\leq N\}\subset\F_p$ be proper and $|B|=o(p^{2/3}).$ Suppose $A=\{a^\prime+b^\prime t:|a^\prime|,|b^\prime|\leq cN^{1/2}\}$ satisfies $AA\subset B.$ Then there is a constant $C$ depending only on $c$ so that when $p$ is sufficiently large, there exists some quadratic polynomial $f\in\Z[x]\setminus\{0\}$ with $H(f)\leq C$ and $\overline{f}(t) = 0$ in 
 ${\mathbb F}_p$.
\end{theorem}
\begin{proof}
	Since $t^2\in AA\subset B,$ we have 
    \begin{equation} t^2=a_0+b_0t\label{1} \end{equation}
 for some $|a_0|,|b_0|\leq N.$ We claim that for any integer $m,n$ with $|m|,|n|\leq 5N,$ $m-nt\neq 0$ in $\F_p.$ Suppose for contradiction there exists $|m_0|,|n_0|\leq 5N$ so that $m_0-n_0t=0.$ Without loss of generality we may assume ${\rm gcd}(m_0,n_0)=1.$ Then (\ref{1}) implies 
 \[(\frac{m_0}{n_0})^2=a_0+\frac{b_0m_0}{n_0} \]
 and hence 
    \begin{equation}m_0^2=a_0n_0^2+b_0m_0n_0\ \text{in}\ \F_p. \label{2} \end{equation}
    By assumption $|B|=(2N+1)^2=o(p^{2/3}),$ when $p$ is sufficiently large we get 
    \[m_0^2<\frac{p}{2},\ |an_0^2+bm_0n_0|<\frac{p}{2}. \]
    Thus (\ref{2}) holds not only in $\F_p$ but also in $\Z.$ This implies $n_0\mid m_0^2$ and hence $|n_0|=1,m_0\equiv\pm t\pmod p$. From the fact that $B$ is proper we also know $|m_0|\geq|t|\geq N+1,$ therefore 
    \[m_0^2\pm b_0m_0\geq N+1>a_0, \] which is a contradiction to (\ref{2}). This finishes the proof of the claim.
	
	Now since $\lambda t^2\in B$ for any $\lambda\in \Pi := \{uv\ :\ |u|,|v| \leq c N^{1/2}\},$ consider 
    \[\{(\lambda+\mu)t^2:\lambda,\mu\in\Pi\}.\] 
    This set will include $wt^2$ for any integer $|w|<c^2 N/2.$ (To see this, take $b = \lfloor cN^{1/2}\rfloor$ and then consider the base-$b$ expansion of all positive integers 
    $\leq c^2N/2$, and note that such integers can be written as $a b + 1\cdot r$, where $0 \leq a \leq cN^{1/2}$ and $0 \leq r \leq cN^{1/2}$.)  
    Suppose $\max(|a_0|,|b_0|)\geq\frac{8}{c^2},$ let 
    \[w=\left\lfloor\frac{2N+1}{\max(|a_0|,|b_0|)}\right\rfloor+1,\]
    it is easy to check that $w<c^2N/2$ when $p$ is large, hence $wt^2\in 2B.$ On the other hand, $2N < \max(|wa_0|,|wb_0|)\leq 3N$. From the claim we know \[wt^2=wa_0+wb_0t\notin 2B, \]a contradiction. Hence $\max(|a_0|,|b_0|)\leq\frac{8}{c^2}$ and equation (\ref{1}) gives the polynomial we want.
\end{proof}
The proof strongly depends on the property that enlarging $B$ by some constant will still give us a proper GAP. We do not know if it remains true for $|B|\geq p^{2/3}$. This motivates us to define a special class of GAPs.
\begin{defn}
    A generalized arithmetic progression $B=x_0+\{\sum_{i=1}^da_ix_i:a_i\in[0,N_i-1]\}$ is called $\kappa$-isolated if there is no nonzero $(a_1,...,a_d) \in {\mathbb Z}^d$ with $|a_i| \leq \kappa N_i$, for all $i=1,...,d$, such that $\sum_ia_ix_i=0.$ 
\end{defn}
Working with isolated GAPs we are able to generalize Theorem 2 to arbitrary rank and product set of two different GAPs. Our main result is the following:
\begin{theorem}
    Let $0<c,c^\prime,\epsilon<1,d\in\N,N=N(p)\in\N$ with $\lim_{p\to \infty}N(p)=\infty$, \[B=\{\sum_{i=1}^da_ix_i:|a_i|\leq N\}\subset\F_p\] be $6$-isolated  with $x_1=1.$ Suppose there exist proper GAPs \[A=\{\sum_{i=1}^da_iy_i:|a_i|\leq cN^{1-\epsilon}\},\ A^\prime=\{\sum_{i=1}^da_iy_i^\prime:|a_i|\leq c^\prime N^{\epsilon}\} \]such that $AA^\prime\subset B.$ Then there is a constant $C=C(c,c^\prime,d)$ so that when $p$ is sufficiently large, for any $i\in[d],$ there exists some polynomial $f_i\in\Z[x]\setminus\{0\}$ of degree at most $d$ with $H(f_i)\leq C$ and $\overline{f}_i(x_i)=0$ in ${\mathbb F}_p$. 
\end{theorem}
To prove Theorem 3, we need the following lemma from linear algebra.
\begin{lem}
    Let $T=(t_{ij})\in M_d(\F_p)$. Suppose there exists $C\geq 1$ so that $|t_{ij}|\leq C$ and $T$ is not invertible. Then there exists a nonzero vector $(a_1,a_2,\ldots,a_d)\in\F_p^d$ with $|a_i|\leq d!C^d$ such that \[(a_1,a_2,\cdots,a_d)T=0.\]
\end{lem}
\begin{proof}
    Let \[\vec{v_i}=(t_{i1},t_{i2},\ldots,t_{id})\in \F_p^d\]
    be the $i$th row vector of $T.$ If $\vec{v_i}=0$ for some $i,$ we can take $a_i=1$ and $a_j=0$ for $j\neq i.$ Hence we may assume $\vec{v_i}\neq 0$ for all $i.$ Without loss of generality, assume $\{\vec{v_1},\ldots,\vec{v_k}\}$ forms a maximal linearly independent set, since $T$ is not invertible we must have $k<d.$ Let $T^\prime\in M_k(\F_p)$ be a submatrix of $T$ using the first $k$ rows so that ${\rm rank}(T^\prime)=k.$ Suppose the columns of $T^\prime$ correspond to the $j_1<j_2<\cdots<j_k$th columns of $T,$ let \[\vec{u}_i=(t_{ij_1},t_{ij_2},\ldots,t_{ij_k})\in\F_p^k,\] 
    We claim that $\vec{u}_{(k+1)}\neq0.$ Indeed, by our choice of $\{v_i\}$, there exists a linear combination 
    \[\vec{v}_{(k+1)}=\sum_{i=1}^kb_i\vec{v}_i,\]
    where the $b_i$'s are not all zero since $\vec{v}_{k+1}\neq 0.$ This implies \[\vec{u}_{(k+1)}=\sum_{i=1}^kb_i\vec{u}_i\]
    and hence $\vec{u}_{k+1}\neq 0.$

    Now by adding the $(k+1)$th row and another column to $T^\prime,$ we obtain a matrix $T^{\prime\prime}\in M_{k+1}(\F_p)$ with ${\rm rank}(T^{\prime\prime})=k$ and all the row vectors of $T''$ are nonzero. It follows that the adjugate matrix ${\rm adj}(T'')$ \footnote{Which is the $(k+1) \times (k+1)$ matrix such that ${\rm adj}(T'')\cdot T'' = {\rm det}(T'') I_{k+1}$, where $I_{k+1}$ is the $(k+1) \times (k+1)$ identity matrix.} of $T''$ has the $(k+1)$th row $(a_1,a_2,\ldots,a_{k+1})$ with $a_{k+1}\neq 0$. Since ${\rm det}(T'')=0$, we get \[\sum_{i=1}^{k+1}a_i\vec{u}_i=0.\] 
    From the fact that $\{\vec{v}_1,\vec{v}_2,\ldots,\vec{v}_k\}$ forms a maximal linearly independent set we can deduce \[\sum_{i=1}^{k+1}a_i\vec{v}_i=0.\] 
    Set $a_i=0$ for all $i>k+1$, since each entry of ${\rm adj}(T'')$ is the determinant of some submatrix of $T'',$ it follows that $|a_i|\leq d!C^d.$ This gives us a nonzero vector $(a_1,\ldots,a_d)\in\F_p^d$ so that \[(a_1,a_2,\cdots,a_d)T=\sum_{i=1}^da_i\vec{v}_i=0.\]
\end{proof}

\begin{proof}[Proof of Theorem 3]
    For any $i,j\in[d],$ 
    \begin{equation}\label{yiyj}
    y_iy_j^\prime=\sum_{k}a_{ij}^{(k)}x_k,\ 
    {\rm where\ } |a_{i,j}^{(k)}| \leq N.
    \end{equation}
    Now since $\lambda y_iy_j^\prime\in B$ for any $\lambda\in \Pi := \{ uv\ :\ |u| \leq c N^{1-\epsilon},\ |v| \leq c' N^\epsilon \}$, consider \[\{(\lambda+\mu)y_iy_j^\prime\ :\ \lambda,\mu\in\Pi\},\] As in the proof of Theorem 2, this set will include $wy_iy_j^\prime$ for any $|w|<c^{\prime\prime} N,$ where $0 < c^{\prime\prime}< 1$ is some constant depending on $c$ and $c^\prime.$ 
    %(To see this, take $b = \lfloor cN^{1-\epsilon}\rfloor$ and then consider the base-$b$ expansion of all positive integers $\leq \lfloor c'N^\epsilon\rfloor b$, and note that such integers can be written as $a b + 1\cdot r$, where $0 \leq a \leq c'N^{\epsilon}$ and $0 \leq r \leq cN^{1-\epsilon}$. 
    This implies that 
    $$
    wy_iy_j^\prime\in AA^\prime+AA^\prime\subset 2B
    $$
    for any $|w|<c^{\prime\prime} N.$ Therefore we have 
    \[w\sum_{k}a_{ij}^{(k)}x_k=\sum_{k}b_w^{(k)}x_k\] 
    where $|b_w^{(k)}|\leq 2N.$ If 
    \[
    \max\{|a_{ij}^{(k)}|:k\in[d]\}\geq\frac{3}{c^{\prime\prime}}, 
    \]
    then for some $|w|<c^{\prime\prime}N$ we would have
    \[
    2N<\max\{|wa_{ij}^{(k)}|:k\in[d]\} \leq 4N, 
    \]
    contradicts with $B$ being $6$-isolated. Hence 
\begin{equation}\label{4}
\max\{|a_{ij}^{(k)}|:k\in[d]\}<\frac{3}{c^{\prime\prime}}.
\end{equation}
    Now fix some $i\in[d],$ and let $\vec{v}=(x_1,x_2,\ldots,x_d)^{T}$, $\vec{u^\prime}=(y_1^\prime,y_2^\prime,\ldots,y_d^\prime)^T\in\F_p^d$.
    Also, let $T\in M_d(\F_p)$ be the matrix whose entry in the $j$th row and $k$th column is
    $$
    T_{j,k}\ =\ a_{ij}^{(k)}
    $$
    so that (\ref{yiyj}) becomes
    \begin{equation} \label{Tveqn}
    T\vec{v}\ =\ y_i \vec{u^\prime}.
    \end{equation}
    From (\ref{4}) it follows that all the entries of the adjugate matrix ${\rm adj}(T)$ are bounded by 
    \[C=(d-1)!3^{d-1}/(c^{\prime\prime})^{d-1}.\] 
    Moreover, $T$ must be invertible since otherwise from Lemma 4, there exists a nonzero vector $(a_1,a_2,\ldots,a_d)\in\F_p^d$ with $|a_i|$ bounded by some constant depending on $c''$ and \[(a_1,a_2,\ldots,a_d)T=0.\]
    Combining this with (\ref{Tveqn}) gives us 
    $$
    y_i\cdot\sum_{j=1}^da_jy_j^\prime=(a_1,a_2,\ldots,a_d)T\vec{v}=0,
    $$
    which would imply that $A^\prime$ is not proper, a contradiction.
    
    Multiplying (\ref{Tveqn}) by ${\rm adj}(T)$ on the left on both sides we get
    $$
    {\rm det}(T) \vec v\ =\ y_i \cdot {\rm adj}(T) \vec{u'}.
    $$
    Taking linear combinations of coordinates on both sides (using coefficients $b_j$), we can deduce
    \[
    A_1 :=\{{\rm det}(T)\sum_jb_jx_j:|b_j|<\frac{c^\prime }{C}N^{\epsilon}\}\subset y_i\cdot A^\prime. 
    \]
    
    Let $\vec{u}=(y_1,y_2,\ldots,y_d)^T$. As before, we define a matrix $T'\in M_d(\F_p)$ where the entry in the $j$th row, $k$th column is
    $$
    T'_{j,k}\ =\ a_{ji}^{(k)}.
    $$
    Again (\ref{yiyj}) becomes 
    \begin{equation}\label{T'veqn}
    T^\prime \vec{v}\ =\ \vec{u} y_i^\prime,
    \end{equation}
    and from (\ref{4}), all the entries of ${\rm adj}(T^\prime)$ are bounded by $C$. Using the same argument that we did earlier for the matrix $T$, we can deduce that $T^\prime$ is also invertible since $A$ is proper.  Multiplying both sides of (\ref{T'veqn}) on the left by ${\rm adj}(T^\prime)$ we get 
    \[
    A_2 :=\{{\rm det}(T^\prime)\sum_jb_jx_j:|b_j|<\frac{c}{C}N^{1-\epsilon}\}\subset A \cdot y_i^\prime. 
    \]
    Hence 
    \[
    A_1A_2\subset y_i\cdot (A' A)\cdot y_i^\prime \subset y_i B y_i^\prime.
    \]
    Now for any $j,k\in[d],$ we have 
    \[
    {\rm det}(T){\rm det}(T') x_jx_k\ =\ y_i\left (\sum_{l=1}^ds_{jk}^{(l)}x_l
    \right ) y_i^\prime
    \]
    and the set 
    \[
    \{{\rm det}(T){\rm det}(T')wx_jx_k:|w|<\frac{cc^\prime}{C^2}N\}
    \]
    will be a subset of $y_i(2B)y_i^\prime$ from the same argument as the one at the beginning of our proof. 
    Again from the fact that $B$ is $6$-isolated we would have 
    \[
    \max\{|s_{jk}^{(l)}|:l\in[d]\}<C^\prime
    \]
    for some $C^\prime=C^\prime(c,c^\prime,d)$. Thus there exists $T_j\in M_d(\F_p)$ whose entries are bounded by $C^\prime$ so that 
    $$
    {\rm det}(T){\rm det}(T') x_j \vec{v}\ = 
    y_i(T_j\vec{v}) y_i^\prime.
    $$
    This equality implies that for any $j\in[d]$, ${\rm det}(T){\rm det}(T^\prime)x_j/(y_iy_i^\prime)$ is an eigenvalue of $T_j$, hence is the root of the characteristic polynomial of $T_j,$ which has degree $d$ and height bounded by some constant depending on $C^\prime.$ From Lemma 4 and the fact that $B$ is proper, $T_1$ must be invertible. Notice that all these eigenvalues share the same eigenvector $\vec{v},$ thus \[({\rm adj}(T_1)T_j)\vec{v}={\rm det}(T_1)T_1^{-1}T_j\vec{v}={\rm det}(T_1)\frac{x_j}{x_1}\vec{v}={\rm det}(T_1) x_j\vec{v}. \] 
    This shows that ${\rm det}(T_1)x_j$ is the root of the characteristic polynomial of ${\rm adj}(T_1)T_j,$ hence $x_j$ is the root of some polynomial of degree at most $d$ with bounded height.
\end{proof}
\begin{rmk}
	The same proof works for non-symmetric GAP $B$ under slightly stronger conditions. Consider $B-B,$ which is a symmetric GAP, and assume the two symmetric GAPs $A-A$ and $A^\prime-A^\prime$ are proper, then $AA^\prime\subset B$ implies that $(A-A)(A^\prime-A^\prime)\subset 2(B-B).$ If $B$ is $24$-isolated, then $2(B-B)$ is $6$-isolated and we can repeat the proof for $2(B-B),A-A$ and $A^\prime-A^\prime.$
\end{rmk}
\begin{rmk}
	If we remove the condition that $x_1=1,$ we can only prove that the ratio $x_i/x_j$ of any two generators $x_i,x_j$ satisfies a low-height, low-degree polynomial equation. This is what we expect since the condition $AA^\prime\subset B$ is not invariant under dilation in general.
\end{rmk}
Since the conclusion of Theorem 3 is about roots of polynomials, which also make sense in the non-commutative ring $M_n(\F_p)$, a natural question is if we consider GAPs in $M_n(\F_p),$ can we get a result similar to Theorem 3? It turns out that under some extra conditions, we can modify the proof of Theorem 3 to get the following:
\begin{cor}
    Let $0<c,c^\prime,\epsilon<1,n,d\in\N,N=N(p)\in\N$ with $\lim_{p\to \infty}N(p)=\infty$, \[B=\{\sum_{i=1}^da_iX_i:|a_i|\leq N\}\subset M_n(\F_p)\] be $6$-isolated  with $X_1=I_n.$ Suppose there exist proper GAPs 
    \[
    A=\{\sum_{i=1}^da_iY_i:|a_i|\leq cN^{1-\epsilon}\},\ A^\prime=\{\sum_{i=1}^da_iY_i^\prime:|a_i|\leq c^\prime N^{\epsilon}\} 
    \]
    such that $AA^\prime\subset B,A^\prime A\subset B.$ Moreover, assume $Y_i,Y_j^\prime$ are invertible for some $i,j\in[d]$. Then there is a constant $C=C(c,c^\prime,d)$  so that when $p$ is sufficiently large, for any $k\in[d],$ there exists some polynomial $f_k\in\Z[x]\setminus\{0\}$ of degree at most $d$ with $H(f_k)\leq C$ and $\overline{f}_k(X_k)=0$ in $M_n(\F_p)$.
\end{cor}
\begin{proof}
    Replacing $x_i,y_i,y_i^\prime$ by $X_i,Y_i,Y_i^\prime$ in the proof of Theorem 3, in the same way we obtain 
    \[A_1\subset Y_i\cdot A^\prime,\ A_2\subset A\cdot Y_j^\prime.\]
    By assumption $A^\prime A\subset B,$ it follows that 
    \[A_1A_2\subset Y_i\cdot B\cdot Y_j^\prime.\]
    Now for any $k\in[d],$ there exists $T_k\in M_d(\F_p)$ with bounded entries so that 
    $$
    {\rm det}(T){\rm det}(T') X_k \vec{v}\ = 
    Y_i(T_k\vec{v}) Y_j^\prime.
    $$
    Since $X_1=I_n,$ $Y_i, Y_j^\prime$ are invertible, by Lemma 4 and the fact that $B$ is proper, we conclude that $T_1$ must be invertible. Moreover, 
    \[T_k\vec{v}\ ={\rm det}(T){\rm det}(T') Y_i^{-1}X_k \vec{v}(Y_j^\prime)^{-1},\ {\rm det}(T){\rm det}(T')T_1^{-1}\vec{v}\ =Y_i\vec{v}Y_j^\prime.\]
    Therefore
    \[({\rm adj}(T_1)T_k)\vec{v}={\rm det}(T_1)T_1^{-1}T_k\vec{v}={\rm det}(T_1) Y_i^{-1}X_kY_i\vec{v}.\]
    Let $f_k(x)\in\Z[x]$ be the characteristic polynomial of ${\rm adj}(T_1)T_k,$ then by Cayley-Hamilton's theorem
    \[\overline{f}_k({\rm det}(T_1) Y_i^{-1}X_kY_i)\vec{v}=\overline{f}_k({\rm adj}(T_1)T_k)\vec{v}=0.\]
    In particular, using the fact $X_1=I_n$ we can deduce 
    \[Y_i^{-1}\overline{f}_k({\rm det}(T_1)X_k)Y_i=\overline{f}_k({\rm det}(T_1) Y_i^{-1}X_kY_i)=0,\]
    which implies \[\overline{f}_k({\rm det}(T_1)X_k)=0.\]
\end{proof}

\section{Open questions}
A natural question is can we weaken the condition on $B$ or $A$ to get the same result as in Theorem 3. A simple example shows that if we remove the isolated condition on $B$ and allow $A$ to be some arbitrary set with $|A|\approx|B|^{1/2}$, then the conclusion of Theorem 3 is no longer true. Let $B=\{a+bN:a,b\in[0,N-1]\}=[0,N^2-1],$ which is a degenerate rank-$2$ GAP. We can take $A=[0,N-1]$ so that $AA\subset B$ and $|A|\approx|B|^{1/2}.$ By choosing $N$ to be about $p^{1/4}$ so that it's not the root of a low-height, low-degree polynomial, we've got a counterexample. From the argument of Theorem 2 we guess the isolated condition could probably be replaced by some restriction on the size of $B.$ 

One can also consider unbalanced GAPs. Suppose \[B=\{\sum_{i=1}^da_ix_i:|a_i|\leq N_i\}\subset\F_p\]and \[A=\{\sum_{i=1}^da_iy_i:|a_i|\leq cM_i\},\ A^\prime=\{\sum_{i=1}^da_iy_i^\prime:|a_i|\leq c^\prime M_i^\prime\}\]with $|A||A^\prime|\approx|B|.$ If we make no restriction on the $N_i$'s, then there is no hope to obtain the same result. For example, $B$ can be constructed with $N_i=1$ for some $i,$ and the corresponding $x_i$ is not the root of some low-height, low-degree polynomial. The question is whether the conclusion of Theorem 3 remains true if we assume \[\lim_{p\rightarrow\infty}N_i=\lim_{p\rightarrow\infty}M_i=\lim_{p\rightarrow\infty}M_i^\prime=\infty \]for each $i\in[d]$ to rule out the extreme case.

\bibliographystyle{abbrv}
\bibliography{bibliography}

\end{document}